\documentclass[reqno]{amsart}
\usepackage{amsmath,amssymb,amsthm,amscd,amsfonts,mathrsfs,verbatim}
\usepackage[all]{xy}

\newcommand{\Z}{{\mathbb Z}}                   
\newcommand{\R}{{\mathbb R}}                   
\newcommand{\C}{{\mathbb C}}                   
\newcommand{\Mod}{{\mathcal M}}               

\newcommand{\CP}[1]{\mathbb{C}P^{#1}}      
\newcommand{\E}{{\mathcal E}}

\DeclareMathOperator{\Dol}{Dol}
\DeclareMathOperator{\Betti}{B}

\DeclareMathOperator{\degen}{deg}

\newtheorem{lem}{Lemma}
\newtheorem{theorem}{Theorem}

\title[Geometric P=W conjecture via plumbing calculus]{The Geometric P=W conjecture in the Painlev\'e cases via plumbing calculus}

\author{Andr\'as N\'emethi}
\address{Alfr\'ed R\'enyi Institute of Mathematics, 1053. Budapest, Re\'altanoda
utca 13-15. Hungary}
\email{nemethi.andras@renyi.hu}

\author{Szil\'ard Szab\'o}
\address{Budapest University of Technology and Economics, 1111. Budapest,
Egry J\'ozsef utca 1. H \'ep\"ulet, Hungary, and
Alfr\'ed R\'enyi Institute of Mathematics, 1053. Budapest, Re\'altanoda
utca 13-15. Hungary}
\email{szabosz@math.bme.hu, szabo.szilard@renyi.hu}

\begin{document}

\begin{abstract}
We use plumbing calculus to prove the homotopy commutativity assertion of the Geometric P=W conjecture in all Painlev\'e cases. 
We discuss the resulting Mixed Hodge structures on Dolbeault and Betti moduli spaces. 
\end{abstract}

\maketitle

\section{Introduction and statement of the main result}\label{sec:Intro}

In this paper we will deal with certain moduli spaces associated to the Painlev\'e equations.
Painlev\'e equations are classically parameterized by a finite set
\begin{equation}\label{eq:I-VI}
   I, II, III(D6), III(D7), III(D8), IV, V_{\degen}, V, VI
\end{equation}
and some continuous parameters. In order to give a uniform treatment of all Painlev\'e cases, throughout we let $X$ stand for any of the symbols~\eqref{eq:I-VI}.
The generic choice of the continuous parameters needed to define the spaces has no relevance to our topological considerations, therefore we omit to spell them out.
We will denote the Dolbeault moduli space associated to the Painlev\'e equation $PX$ by $\Mod_{\Dol}^X$; likewise, we denote the associated wild character variety
(or Betti moduli space) by ${\Mod}_{\Betti}^X$.
The  Dolbeault moduli space $\Mod_{\Dol}^X$ parameterizes certain singular \emph{Higgs bundles} $(\E, \theta )$ of rank $2$ over $\CP1$, with logarithmic or irregular
singularities, up to S-equivalence.
The definitions of these spaces depend on choices of irregular type, of eigenvalues of the residues of the Higgs field at the punctures
and of parabolic weights; we choose these weights generically so that the moduli spaces be smooth manifolds.
For a more detailed description of the moduli-theoretic definition of these spaces, see~\cite{ISS1,ISS2,ISS3,Sz_BNR,Sz_PW} and the references therein.
On the other hand, ${\Mod}_{\Betti}^X$ parameterizes certain monodromy and Stokes data, up to overall conjugation.
Again, the definitions of these spaces depend on choices of some parameters (eigenvalues of the local monodromy matrices and weights).
These spaces have been studied in detail by M.~van~der~Put and M.~Saito~\cite{PS}; for a general treatment of wild character varieties, see work of P.~Boalch~\cite{Boalch1,Boalch3}.

Now, to the parameters fixed to define $\Mod_{\Dol}^X$ there correspond parameters on the Betti side by Simpson's table~\cite{Sim_Hodge} and its extension to the irregular case by Biquard and Boalch~\cite{BB}.
It follows from~\cite{BB} coupled with the irregular Riemann--Hilbert correspondence that for corresponding choices, the Dolbeault and Betti spaces are diffeomorphic smooth complex surfaces
$$
  \psi: \Mod_{\Dol}^X \to {\Mod}_{\Betti}^X.
$$

It follows from~\cite{ISS1,ISS2,ISS3} that the Dolbeault moduli spaces naturally come equipped with a proper map
\begin{align}
   h: \Mod_{\Dol}^X & \to \C \label{eq:Hitchin} \\
   (\E, \theta ) & \mapsto \det (\theta ) \notag
\end{align}
called \emph{Hitchin map}, which in the cases at hand is an elliptic fibration. 
The original properness result of N.~Hitchin~\cite{Hit_SBIS} was generalized to the singular case by K.~Yokogawa in~\cite[Corollary~5.2]{Y}. 
For $R \in \R$ we denote by $B_R(0)$ the open disk of radius $R$ centered at $0$ in the base of the Hitchin map.

For a simplicial complex $\mathcal{N}$, we denote by $| \mathcal{N} |$ the associated topological space.
It follows from~\cite{PS} that there exists a smooth compactification $\widetilde{\Mod}_{\Betti}^X$ of
$\Mod_{\Betti}^X$ by a simple normal crossing divisor $\tilde{D}_{\infty}^X$ with smooth components and with nerve complex $\mathcal{N}^X$ such that
$| \mathcal{N}^X |$ is homotopy equivalent to $S^1$.
The $0$-skeleton $\mathcal{N}_0^X$ is in bijection with irreducible components of $\tilde{D}_{\infty}^X$, the edges in the $1$-skeleton $\mathcal{N}_1^X$
correspond to intersection points between the corresponding divisor components.
Let $T_j^X \subset {\Mod}_{\Betti}^X$ denote a tubular neighbourhood of the $j$'th irreducible component of $\tilde{D}_{\infty}^X$ and set
$$
  T^X = \cup_{j \in \mathcal{N}_0^X}\ T_j^X ,
$$
which is a neighbourhood of $\tilde{D}_{\infty}^X$.
Let $\{ \phi_j^X \}_{j \in \mathcal{N}_0^X}$ be a partition of unity on $T^X$ dominated by the covering $\{ T_j^X \}_{j \in \mathcal{N}_0^X}$.
In~\cite{Sim}, C.\;Simpson defines the continuous map
$$
  \phi = (\phi_j)_{j \in \mathcal{N}_0^X}: T^X \to \R^{\mathcal{N}_0^X}
$$
whose image is the associated topological space of a simplicial complex isomorphic to $\mathcal{N}^X$.

\begin{theorem}\label{thm:Simpson}
 For all sufficiently large $R \in \R$, there exists a homotopy commutative square
 $$
  \xymatrix{\Mod_{\Dol}^X \setminus h^{-1} (B_R(0))  \ar[d]_{h} \ar[r]^{\hspace*{6mm}\psi} & {\Mod}_{\Betti}^X \cap T^X \ar[d]^{\phi} \\
  \C \setminus B_R(0) \ar[r] & | \mathcal{N}^X |. }
 $$
\end{theorem}

This statement in higher generality was conjectured by L.\;Katzarkov, A.\;Noll, P.\;Pandit and C.\;Simpson~\cite[Conjecture 1.1]{KNPS}
and subsequently named Geometric $P = W$ conjecture by C.\;Simpson~\cite[Conjecture 11.1]{Sim}.
Previously, in~\cite{HdCM} de Cataldo, Hausel and Migliorini had conjectured a relationship, called  $P = W$ conjecture, between the perverse Leray filtration induced by the Hitchin map on the cohomology of Dolbeault moduli spaces
and Deligne's weight filtration on the cohomology of Betti moduli spaces. The Geometric $P = W$ conjecture is supposed to provide a geometric explanation of the $P = W$ conjecture on the lowest weight graded pieces.
As A.~Harder~\cite{Harder} points out, it would follow from a conjecture of D.~Auroux~\cite[Conjecture~7.3]{Auroux} that such an implication indeed exists, namely the Geometric $P = W$ conjecture does imply the lowest weight graded part
of the original $P = W$ conjecture.
In~\cite{Sz_PW,Sz_abelianization} the second author proved the above homotopy commutativity assertion for the Painlev\'e VI case using asymptotic abelianization of solutions of Hitchin's equations,
and proved that for Painlev\'e moduli spaces the Geometric $P = W$ conjecture implies the original $P = W$ conjecture.
Here, we offer a different approach to homotopy commutativity which works in all Painlev\'e cases, by employing the plumbing calculus developped by W.~Neumann \cite{Neu}.


\noindent {\bf Acknowledgements:}
 During the preparation of this manuscript both authors were
  supported by the \emph{Lend\"ulet} Low Dimensional Topology
  grant of the Hungarian Academy of Sciences and by the grant KKP126683 of NKFIH.

\section{Structure of the spaces}

In this section we recall from~\cite{Sz_PW} some relevant results about the structure of the moduli spaces in question.

\subsection{Dolbeault space}\label{subsec:Dolbeault}

According to \cite[Lemma~2]{Sz_PW} there exists an elliptic fibration
$$
  \tilde{h}: E(1) \to \CP1
$$
on the rational elliptic surface
\begin{equation}\label{eq:elliptic_surface}
   E(1) = \CP2 \# 9 \overline{\CP{}}^2
\end{equation}
and an embedding
$$
  \Mod_{\Dol}^{X} \hookrightarrow E(1); \  \Mod_{\Dol}^{X} =E(1)\setminus \tilde{h}^{-1}(\infty)
$$
so that
$$
  h = \tilde{h}|_{\Mod_{\Dol}^{X}}.
$$
We denote the fiber at infinity of $\tilde{h}$ by
$$
  F_{\infty}^{X} = \tilde{h}^{-1} (\infty ).
$$
The type of the curves $F_{\infty}^{X}$ is determined by $X$ and is listed in the second column of Table~\ref{table}.
It follows from this analysis that a punctured neighbourhood $\Mod_{\Dol}^X \setminus h^{-1} (B_R(0))$ of $F_{\infty}^{X}$ in $E(1)$ has
the homotopy type of a $T^2$-fibration $Y^X_{\Dol}$ over the circle $S^1_R$ of radius $R\gg 0$.

\subsection{Betti space}\label{subsec:Betti}

It is shown in \cite{PS} that for all $X$ the space $\Mod_{\Betti}^X$ is a smooth affine cubic surface defined by a polynomial
\begin{equation}\label{eq:cubic}
  f^X (x_1, x_2, x_3) = x_1 x_2 x_3 + Q^X (x_1, x_2, x_3)
\end{equation}
for some (not necessarily homogeneous) affine quadric $Q^X$.
We consider the homogeneous cubic polynomial $F^X$ in the homogeneous coordinates $[X_0 : X_1 : X_2 : X_3]$ associated with  $f^X$ and the projective compactification
\begin{equation}\label{eq:compactifcation_Betti}
 \overline{\Mod}_{\Betti}^X = (F^X) \in \CP3
\end{equation}
of $\Mod_{\Betti}^X$. The intersection of $\overline{\Mod}_{\Betti}^X$ with the plane at infinity $\CP2_{\infty}$ defined by $X_0 = 0$ has equation
$$
  X_1 X_2 X_3 = 0,
$$
hence the divisor at infinity is the union
$$
  \bar{D}_{\Betti}^X = L_1 \cup L_2 \cup L_3
$$
of three lines $L_i = \{X_i=0\}  \subset \CP2_{\infty}$ in general position.
It is shown in \cite[Proposition~2]{Sz_PW} that some of the pairwise intersection points $L_i \cap L_{i'}$ are singular points of $\overline{\Mod}_{\Betti}^X$, and their singularity types are computed.
We included the sets of singularity types in the third column of Table \ref{table}, with the singularities at the various points separated by a $+$ sign.
For instance, the symbol $A_1 + A_1$ in the case $X=IV$ means that two of the intersection points are $A_1$ singular points of $\overline{\Mod}_{\Betti}^X$, and the third one is a smooth point.
We define
\begin{equation}\label{eq:min_res_sing}
   \sigma: \widetilde{\Mod}_{\Betti}^X \to \overline{\Mod}_{\Betti}^X
\end{equation}
as the minimal resolution of singularities. Let
$$
  \widetilde{D}_{\Betti}^X = \sigma^{-1} \bar{D}_{\Betti}^X
$$
stand for the total transform of $\bar{D}_{\Betti}^X$ under $\sigma$. Let us denote by $L_i^{\sigma}$ the proper transform of the line $L_i$ under $\sigma$.

\begin{table}
\begin{center}
\begin{tabular}{|l|l|r|}
 \hline
 $X$ & $F_{\infty}^X$ & $\mbox{Sing} \left( \overline{\Mod}_{\Betti}^X \right)$ \\
 \hline
 \hline
 $VI$ & $D_4^{(1)}$ & $\emptyset$ \\
 \hline
 $V$ & $D_5^{(1)}$ & $A_1$  \\
 \hline
 $V_{\degen}$ & $D_6^{(1)}$ & $A_2$ \\
 \hline
 $III(D6)$ & $D_6^{(1)}$ & $A_2$ \\
 \hline
 $III(D7)$ & $D_7^{(1)}$ & $A_3$ \\
 \hline
 $III(D8)$ & $D_8^{(1)}$ & $A_4$ \\
 \hline
 $IV$ & $E_6^{(1)}$ & $A_1 + A_1$ \\
 \hline
 $II$ & $E_7^{(1)}$ & $A_1 + A_1 + A_1$ \\
 \hline
 $I$ & $E_8^{(1)}$ & $A_2 + A_1 + A_1$ \\
 \hline
\end{tabular}
\end{center}
 \caption{Fiber at infinity of $\Mod_{\Dol}^X$ and singularities of $\overline{\Mod}_{\Betti}^X$}
 \label{table}
\end{table}

\begin{lem}
A punctured neighbourhood ${\Mod}_{\Betti}^X \cap T^X$ of $\widetilde{D}_{\Betti}^X$ in $\widetilde{\Mod}_{\Betti}^X$ has the homotopy type of ${\Mod}_{\Betti}^X \partial \cap T^X$,
which is a plumbed $3$-manifold $Y^X_{\Betti}$ over the following decorated graphs.
The vertices labelled by $-1$ correspond to $L_1^{\sigma}, L_2^{\sigma}, L_3^{\sigma}$, and the
vertices labelled by $-2$ correspond to the exceptional curves of the $A_*$--singularities of $\overline{\Mod}_{\Betti}^X$.
\end{lem}

The underlying graph is what we have denoted by $\mathcal{N}^X$ in Section~\ref{sec:Intro}, and the label of the $i$'th vertex is the Euler-number $-b_i^X$ of the $S^1$-fibration attached
to the corresponding component, used for the plumbing.

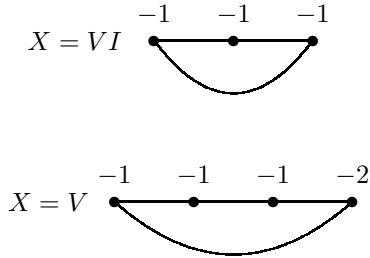
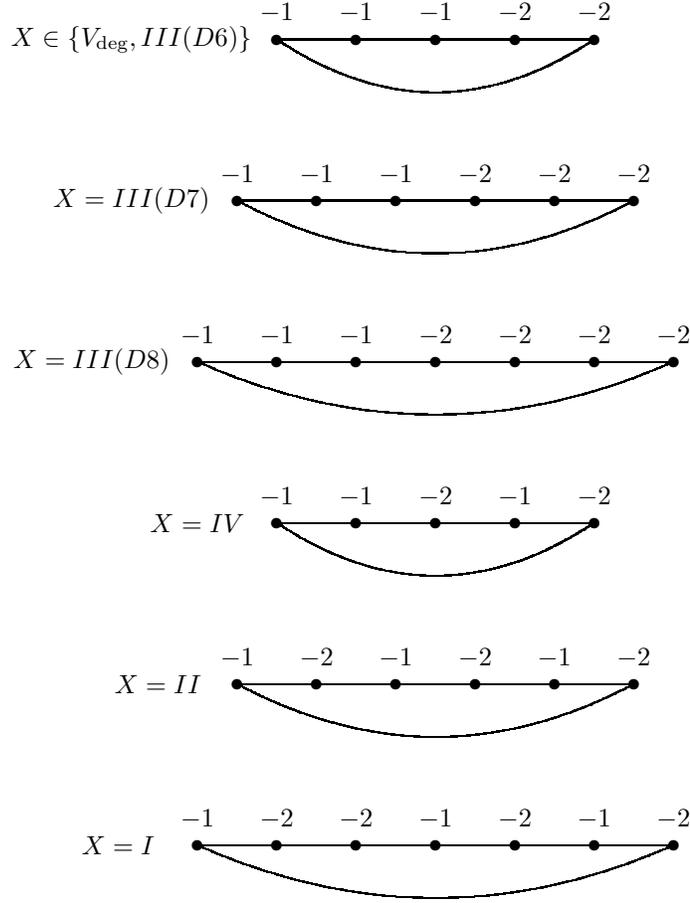
\begin{picture}(100,60)(-90,0) 

\put(0,40){\makebox(0,0){$X=VI$}}

\put(30,40){\circle*{4}}        
\put(60,40){\circle*{4}}
\put(90,40){\circle*{4}}

\put(30,40){\line(1,0){30}}                      
\put(60,40){\line(1,0){30}}
\qbezier(30,40)(60,0)(90,40)

\put(30,50){\makebox(0,0){$-1$}}
\put(60,50){\makebox(0,0){$-1$}}
\put(90,50){\makebox(0,0){$-1$}}

\end{picture}

\begin{picture}(100,60)(-75,0) 

\put(5,40){\makebox(0,0){$X=V$}}

\put(30,40){\circle*{4}}        
\put(60,40){\circle*{4}}
\put(90,40){\circle*{4}}
\put(120,40){\circle*{4}}

\put(30,40){\line(1,0){30}}                      
\put(60,40){\line(1,0){30}}
\put(90,40){\line(1,0){30}}
\qbezier(30,40)(75,0)(120,40)

\put(30,50){\makebox(0,0){$-1$}}
\put(60,50){\makebox(0,0){$-1$}}
\put(90,50){\makebox(0,0){$-1$}}
\put(120,50){\makebox(0,0){$-2$}}

\end{picture}

\begin{picture}(100,60)(-60,0) 

\put(-25,40){\makebox(0,0){$X\in\{V_{\degen}, III(D6) \}$}}

\put(30,40){\circle*{4}}        
\put(60,40){\circle*{4}}
\put(90,40){\circle*{4}}
\put(120,40){\circle*{4}}
\put(150,40){\circle*{4}}

\put(30,40){\line(1,0){30}}                      
\put(60,40){\line(1,0){30}}
\put(90,40){\line(1,0){30}}
\put(120,40){\line(1,0){30}}
\qbezier(30,40)(90,0)(150,40)

\put(30,50){\makebox(0,0){$-1$}}
\put(60,50){\makebox(0,0){$-1$}}
\put(90,50){\makebox(0,0){$-1$}}
\put(120,50){\makebox(0,0){$-2$}}
\put(150,50){\makebox(0,0){$-2$}}

\end{picture}

\begin{picture}(100,60)(-45,0) 

\put(-10,40){\makebox(0,0){$X = III(D7)$}}

\put(30,40){\circle*{4}}        
\put(60,40){\circle*{4}}
\put(90,40){\circle*{4}}
\put(120,40){\circle*{4}}
\put(150,40){\circle*{4}}
\put(180,40){\circle*{4}}

\put(30,40){\line(1,0){30}}                      
\put(60,40){\line(1,0){30}}
\put(90,40){\line(1,0){30}}
\put(120,40){\line(1,0){30}}
\put(150,40){\line(1,0){30}}
\qbezier(30,40)(105,0)(180,40)

\put(30,50){\makebox(0,0){$-1$}}
\put(60,50){\makebox(0,0){$-1$}}
\put(90,50){\makebox(0,0){$-1$}}
\put(120,50){\makebox(0,0){$-2$}}
\put(150,50){\makebox(0,0){$-2$}}
\put(180,50){\makebox(0,0){$-2$}}

\end{picture}

\begin{picture}(100,60)(-30,0) 

\put(-10,40){\makebox(0,0){$X = III(D8)$}}

\put(30,40){\circle*{4}}        
\put(60,40){\circle*{4}}
\put(90,40){\circle*{4}}
\put(120,40){\circle*{4}}
\put(150,40){\circle*{4}}
\put(180,40){\circle*{4}}
\put(210,40){\circle*{4}}

\put(30,40){\line(1,0){30}}                      
\put(60,40){\line(1,0){30}}
\put(90,40){\line(1,0){30}}
\put(120,40){\line(1,0){30}}
\put(150,40){\line(1,0){30}}
\put(180,40){\line(1,0){30}}
\qbezier(30,40)(120,0)(210,40)

\put(30,50){\makebox(0,0){$-1$}}
\put(60,50){\makebox(0,0){$-1$}}
\put(90,50){\makebox(0,0){$-1$}}
\put(120,50){\makebox(0,0){$-2$}}
\put(150,50){\makebox(0,0){$-2$}}
\put(180,50){\makebox(0,0){$-2$}}
\put(210,50){\makebox(0,0){$-2$}}

\end{picture}

\begin{picture}(100,60)(-60,0) 

\put(0,40){\makebox(0,0){$X=IV$}}

\put(30,40){\circle*{4}}        
\put(60,40){\circle*{4}}
\put(90,40){\circle*{4}}
\put(120,40){\circle*{4}}
\put(150,40){\circle*{4}}

\put(30,40){\line(1,0){30}}                      
\put(60,40){\line(1,0){30}}
\put(90,40){\line(1,0){30}}
\put(120,40){\line(1,0){30}}
\qbezier(30,40)(90,0)(150,40)

\put(30,50){\makebox(0,0){$-1$}}
\put(60,50){\makebox(0,0){$-1$}}
\put(90,50){\makebox(0,0){$-2$}}
\put(120,50){\makebox(0,0){$-1$}}
\put(150,50){\makebox(0,0){$-2$}}

\end{picture}

\begin{picture}(100,60)(-45,0) 

\put(0,40){\makebox(0,0){$X=II$}}

\put(30,40){\circle*{4}}        
\put(60,40){\circle*{4}}
\put(90,40){\circle*{4}}
\put(120,40){\circle*{4}}
\put(150,40){\circle*{4}}
\put(180,40){\circle*{4}}

\put(30,40){\line(1,0){30}}                      
\put(60,40){\line(1,0){30}}
\put(90,40){\line(1,0){30}}
\put(120,40){\line(1,0){30}}
\put(150,40){\line(1,0){30}}
\qbezier(30,40)(105,0)(180,40)

\put(30,50){\makebox(0,0){$-1$}}
\put(60,50){\makebox(0,0){$-2$}}
\put(90,50){\makebox(0,0){$-1$}}
\put(120,50){\makebox(0,0){$-2$}}
\put(150,50){\makebox(0,0){$-1$}}
\put(180,50){\makebox(0,0){$-2$}}

\end{picture}

\begin{picture}(100,60)(-30,0) 

\put(0,40){\makebox(0,0){$X=I$}}

\put(30,40){\circle*{4}}        
\put(60,40){\circle*{4}}
\put(90,40){\circle*{4}}
\put(120,40){\circle*{4}}
\put(150,40){\circle*{4}}
\put(180,40){\circle*{4}}
\put(210,40){\circle*{4}}

\put(30,40){\line(1,0){30}}                      
\put(60,40){\line(1,0){30}}
\put(90,40){\line(1,0){30}}
\put(120,40){\line(1,0){30}}
\put(150,40){\line(1,0){30}}
\put(180,40){\line(1,0){30}}
\qbezier(30,40)(120,0)(210,40)

\put(30,50){\makebox(0,0){$-1$}}
\put(60,50){\makebox(0,0){$-2$}}
\put(90,50){\makebox(0,0){$-2$}}
\put(120,50){\makebox(0,0){$-1$}}
\put(150,50){\makebox(0,0){$-2$}}
\put(180,50){\makebox(0,0){$-1$}}
\put(210,50){\makebox(0,0){$-2$}}

\end{picture}

\begin{proof}
Let $\ell \subset \CP2_{\infty}$ be a generic line, and consider the pencil of projective planes in $\CP3$ containing $\ell$.
This pencil exhibits $\overline{\Mod}_{\Betti}^X$ as a pencil of cubic curves, with base locus
$$
  B = (L_1 \cap \ell ) \cup (L_2 \cap \ell ) \cup (L_3 \cap \ell ).
$$
By genericity, each line $L_i$ contains exactly one point of $B$, and these points are not points of the form $L_i \cap L_{i'}$ for $i \neq i'$.
The total space $\overline{\Mod}_{\Betti}^X$ of the pencil is singular, and~\eqref{eq:min_res_sing} is its minimal resolution. It follows that
$\widetilde{\Mod}_{\Betti}^X$ is a smooth elliptic pencil with base locus
$$
  B = (L_1^{\sigma} \cap \ell ) \cup (L_2^{\sigma} \cap \ell ) \cup (L_3^{\sigma} \cap \ell ).
$$
 Let
$$
  \xymatrix{E \ar[d] \ar[r]^{\omega} & \widetilde{\Mod}_{\Betti}^{PX}  \\
  \CP1 & }
$$
be the associated elliptic fibration, obtained by blowing up $B$ in $\widetilde{\Mod}_{\Betti}^X$.
Let us denote by $L_i^{\omega}$ the proper transform of $L_i^{\sigma}$ under $\omega$.
From the configuration of $L_1, L_2, L_3$ we know that the fiber at infinity of $E$ contains $L_i^{\omega}$ for each $i$, and these
components are part of a cycle of rational curves.
According to Kodaira's list~\cite{Kod}, the fiber at infinity is of type $I_N$ for some $3 \leq N$, i.e. a cycle of length $N$ with all self-intersection numbers equal to $-2$.
In particular, we have
$$
  ( L_i^{\omega}, L_i^{\omega} ) = -2.
$$
Now, as $L_i^{\omega}$ is the proper transform of $L_i^{\sigma}$ under blow-up at a single point, we infer that
$$
  ( L_i^{\sigma}, L_i^{\sigma} ) = -1.
$$
The components of the $I_N$ singular fiber other than $L_1^{\sigma}, L_2^{\sigma}, L_3^{\sigma}$ are exceptional divisors of $\sigma$,
so $B$ is disjoint from them.
This shows that their self-intersection numbers in $\widetilde{\Mod}_{\Betti}^{PX}$ are equal to their self-intersection numbers in $E$, namely $-2$.
The lenghts of the chains of $-2$ curves in the plumbing graph are therefore equal to the Milnor numbers of the singular points of the given
$\overline{\Mod}_{\Betti}^X$, as given in Table~\ref{table}. This finishes the proof.
\end{proof}
Another proof can be done via a case by case study of the local equations of the
surface germs at the intersection points of the three lines at infinity
(for their equations see \cite[Sections~3.1-3.9]{Sz_PW}).

\section{Plumbing calculus}\label{sec:plumbing}

Oriented plumbed $3$-manifolds arise from $S^1$-bundles over smooth real surfaces, identified along the intersection pattern given by a decorated graph (called plumbing graph). Let us provide some details
(for more see e.g. \cite{Neu}).

The vertices of the graphs are in bijection with compact oriented (real) smooth surfaces.
Each vertex is  decorated by the Euler number of the $S^1$-bundle over the corresponding surface that is  used
in the plumbing construction.
The genus of all components relevant to us is $0$ (and the surfaces admit no boundary components), hence we omit all
such decorations (which in general theory might appear).
The edges between vertices correspond to plumbing operations of the corresponding $S^1$-bundles.
If $(v,u)$ is an edge connecting the vertices $v$ and $u$, then we fix small discs $D_v$ (resp. $D_u$)
in the surface $S_v$ and $S_u$ corresponding to $v$ and $u$, we fix trivializations of the bundles over them,
we delete their interiors, and we identify  the boundaries (both tori) by  the following pattern.
The edges  are decorated by a sign. A positive edge means  an identification of the  the base-fiber $S^1$ factors of $v$ side with the fiber-base
factors of the $u$ side by keeping their orientation.
A negative edge means the same identifications with opposite orientations.
If the sign of an edge is missing then it is a positive one.

If $C$ is a normal crossing compact complex curve (with irreducible components $\{C_i\}_i$)
in a smooth complex surface, and $Y$ is the boundary of
a small tubular neighbourhood of $C$, then $Y$ admits a plumbing graph (the `dual graph' of $C$)
 such that the vertices correspond to the irreducible components, the Euler decorations are the Euler numbers of the
 normal bundles of the embedded curve components, and the edges (all decorated by positive sign) correspond to the
 intersection points of the curve components.

Usually different graphs might produce diffeomorphic 3-manifolds.
In \cite{Neu} W.~Neumann has given a list of moves R1--R8, R0', R2/4 on decorated graphs such that if two oriented plumbed 3-manifolds are  diffeomorphic
by an orientation preserving diffeomorphism then the corresponding plumbing graphs
can be connected by a finite sequence of moves.
In this section we provide the moves on graphs that will be of use to us in order to identify the oriented plumbed $3$-manifolds $\partial (\Mod_{\Dol}^X \setminus h^{-1} (B_R(0)))$
and $\partial ({\Mod}_{\Betti}^X \cap T^X )$. Let us emphasize that we do not give an exhaustive list of Neumann's moves, instead we provide here only  those  ones
that  will be used here in our proofs.

The first moves we describe are two different versions of blow-down, labelled R1 in \cite[Section~2]{Neu}, and O2 in \cite[Chapter~2]{Nemethi_konyv}:

\begin{picture}(100,40)(5,10)

\put(40,20){\circle*{4}} \put(80,20){\circle*{4}}
\put(120,20){\circle*{4}}                           

\put(40,20){\line(1,0){80}}                        

\put(40,20){\line(-2,-1){30}}                      
\put(40,20){\line(-2, 1){30}}
\put(20,25){\makebox(0,0){$\cdot$}}                
\put(20,20){\makebox(0,0){$\cdot$}}
\put(20,15){\makebox(0,0){$\cdot$}}

\put(120,20){\line(2,-1){30}}                      
\put(120,20){\line(2, 1){30}}
\put(135,25){\makebox(0,0){$\cdot$}}                
\put(135,20){\makebox(0,0){$\cdot$}}
\put(135,15){\makebox(0,0){$\cdot$}}

\put(42,28){\makebox(0,0){$e_i$}}                  
\put(80,29){\makebox(0,0){$\epsilon$}}
\put(118,28){\makebox(0,0){$e_j$}}


\put(170,20){\vector(1,0){20}}                     
\put(170,20){\vector(-1,0){20}}
\put(170,30){\makebox(0,0){R1}}

\put(240,20){\circle*{4}}                    
\put(242,28){\makebox(0,0){$e_i-\epsilon$}}  

\put(240,20){\line(-2,-1){30}}                      
\put(240,20){\line(-2, 1){30}}

\put(220,25){\makebox(0,0){$\cdot$}}                
\put(220,20){\makebox(0,0){$\cdot$}}
\put(220,15){\makebox(0,0){$\cdot$}}

\put(300,20){\circle*{4}} 
\put(298,28){\makebox(0,0){$e_j-\epsilon$}}  

\put(300,20){\line(2,-1){30}}                      
\put(300,20){\line(2, 1){30}}
\put(315,25){\makebox(0,0){$\cdot$}}                
\put(315,20){\makebox(0,0){$\cdot$}}
\put(315,15){\makebox(0,0){$\cdot$}}

\put(240,20){\line(1,0){60}} 
\put(270,26){\makebox(0,0){$-\epsilon$}}      
\put(265,20){\makebox(0,0){$\times$}}   

\end{picture}

\begin{picture}(100,50)(5,0)
\put(40,20){\circle*{4}} \put(120,20){\circle*{4}}   

\put(40,20){\line(-2,-1){30}}                      
\put(40,20){\line(-2, 1){30}}
\put(20,25){\makebox(0,0){$\cdot$}}                
\put(20,20){\makebox(0,0){$\cdot$}}
\put(20,15){\makebox(0,0){$\cdot$}}

\put(42,28){\makebox(0,0){$e_i$}}                  
\put(120,29){\makebox(0,0){$\epsilon$}}

\qbezier(40,20)(80,30)(120,20)
\qbezier(40,20)(80,10)(120,20)



\put(170,20){\vector(1,0){20}}                     
\put(170,20){\vector(-1,0){20}}
\put(170,30){\makebox(0,0){R1}}


\put(240,20){\line(-2,1){30}}       
\put(240,20){\line(-2,-1){30}}
\put(240,20){\circle*{4}}
\put(220,24){\makebox(0,0){$\vdots$}}
\put(240,32){\makebox(0,0){$e_i-2\epsilon$}}

\put(240,20){\line(2,1){20}} \put(240,20){\line(2,-1){20}}
\qbezier(260,30)(290,40)(293,20) \qbezier(260,10)(290,0)(293,20)

\put(300,20){\makebox(0,0){$-\epsilon$}}
\put(293,20){\makebox(0,0){$\times$}}
\end{picture}

\noindent where $\epsilon=\pm 1$, and the cross indicates the edge  where we apply the move (this will be helpful
in the calculus below).

The second move we need is labelled R2/4 in \cite[Section~3]{Neu}, and O7 in \cite[Chapter~2]{Nemethi_konyv};
it is the composition of one unoriented handle absorption and two $\mathbb{R}P^2$-extrusions:

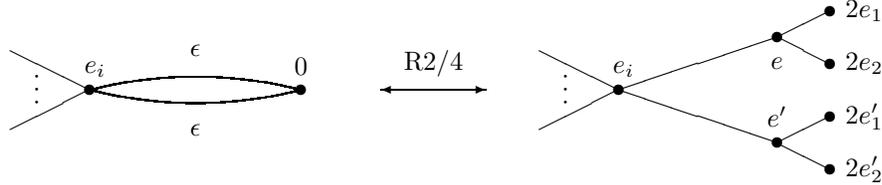
\begin{picture}(100,70)(0,-15)
\put(40,20){\circle*{4}} \put(120,20){\circle*{4}}   

\put(40,20){\line(-2,-1){30}}                      
\put(40,20){\line(-2, 1){30}}
\put(20,25){\makebox(0,0){$\cdot$}}                
\put(20,20){\makebox(0,0){$\cdot$}}
\put(20,15){\makebox(0,0){$\cdot$}}

\put(42,28){\makebox(0,0){$e_i$}}                  
\put(120,29){\makebox(0,0){$0$}}

\qbezier(40,20)(80,30)(120,20)
\qbezier(40,20)(80,10)(120,20)

\put(80,35){\makebox(0,0){$\epsilon$}}
\put(80,5){\makebox(0,0){$\epsilon $}}


\put(170,20){\vector(1,0){20}}                     
\put(170,20){\vector(-1,0){20}}
\put(170,30){\makebox(0,0){R2/4}}


\put(240,20){\circle*{4}}                    
\put(242,28){\makebox(0,0){$e_i$}}  

\put(240,20){\line(-2,-1){30}}                      
\put(240,20){\line(-2, 1){30}}

\put(220,25){\makebox(0,0){$\cdot$}}                
\put(220,20){\makebox(0,0){$\cdot$}}
\put(220,15){\makebox(0,0){$\cdot$}}

\put(300,30){\makebox(0,0){$e$}} \put(300,10){\makebox(0,0){$e'$}} 
\put(333,10){\makebox(0,0){$2e_1'$}}
\put(333,50){\makebox(0,0){$2e_1$}}
\put(333,30){\makebox(0,0){$2e_2$}}
\put(333,-10){\makebox(0,0){$2e_2'$}}
\put(300,40){\line(2,1){20}}                      
 \put(300,40){\line(2,-1){20}}
    \put(300,0){\line(2,1){20}}                      
 \put(300,0){\line(2,-1){20}}                   
\put(240,20){\line(3,1){60}} \put(240,20){\line(3,-1){60}}
\put(300,40){\circle*{4}}\put(300,0){\circle*{4}}
\put(320,50){\circle*{4}}\put(320,30){\circle*{4}}
\put(320,10){\circle*{4}}\put(320,-10){\circle*{4}}

\end{picture}

\noindent where $\epsilon=\pm 1$, $e_1,e_2, e_1',e_2'\in\{-1,+1\}$ and
$e=(e_1+e_2)/2$, $e'=(e'_1+e'_2)/2$.

The last relevant moves are three versions of Seifert graph exchange labelled R7 in  \cite[Section~2]{Neu}, and O6 in \cite[Chapter~2]{Nemethi_konyv}:

\begin{picture}(100,100)

\put(35,35){\circle*{4}} 
\put(35,47){\makebox(0,0){$-1$}}
\qbezier (35,35) (85,55) (88,35) 
\qbezier (35,35) (85,15) (88,35) 
\put(95,35){\makebox(0,0){$-$}}

\put(130,35){\vector(1,0){20}}                     
\put(130,35){\vector(-1,0){20}}
\put(130,45){\makebox(0,0){R7}}

\put(180,35){\circle*{4}}                         
\put(210,35){\circle*{4}}
\put(240,35){\circle*{4}}
\put(270,35){\circle*{4}}
\put(300,35){\circle*{4}}
\put(240,65){\circle*{4}}
\put(240,95){\circle*{4}}

\put(180,35){\line(1,0){30}}                      
\put(210,35){\line(1,0){30}}
\put(240,35){\line(1,0){30}}
\put(270,35){\line(1,0){30}}
\put(240,35){\line(0,1){30}}
\put(240,65){\line(0,1){30}}

\put(180,45){\makebox(0,0){$-2$}}                 
\put(210,45){\makebox(0,0){$-2$}}
\put(230,45){\makebox(0,0){$-2$}}
\put(270,45){\makebox(0,0){$-2$}}
\put(300,45){\makebox(0,0){$-2$}}
\put(230,65){\makebox(0,0){$-2$}}
\put(230,95){\makebox(0,0){$-2$}}

\end{picture}

\begin{picture}(100,70)(20,0)

\put(35,35){\circle*{4}} 
\put(35,47){\makebox(0,0){$0$}}
\qbezier (35,35) (85,55) (88,35) 
\qbezier (35,35) (85,15) (88,35) 
\put(95,35){\makebox(0,0){$-$}}

\put(130,35){\vector(1,0){20}}                     
\put(130,35){\vector(-1,0){20}}
\put(130,45){\makebox(0,0){R7}}

\put(180,35){\circle*{4}}                         
\put(210,35){\circle*{4}}
\put(240,35){\circle*{4}}
\put(270,35){\circle*{4}}
\put(300,35){\circle*{4}}
\put(330,35){\circle*{4}}
\put(360,35){\circle*{4}}
\put(270,65){\circle*{4}}

\put(180,35){\line(1,0){30}}                      
\put(210,35){\line(1,0){30}}
\put(240,35){\line(1,0){30}}
\put(270,35){\line(1,0){30}}
\put(300,35){\line(1,0){30}}
\put(330,35){\line(1,0){30}}
\put(270,35){\line(0,1){30}}

\put(180,45){\makebox(0,0){$-2$}}                 
\put(210,45){\makebox(0,0){$-2$}}
\put(240,45){\makebox(0,0){$-2$}}
\put(260,45){\makebox(0,0){$-2$}}
\put(300,45){\makebox(0,0){$-2$}}
\put(330,45){\makebox(0,0){$-2$}}
\put(360,45){\makebox(0,0){$-2$}}
\put(260,65){\makebox(0,0){$-2$}}

\end{picture}

\begin{picture}(100,70)(40,0)

\put(35,35){\circle*{4}} 
\put(35,47){\makebox(0,0){$1$}}
\qbezier (35,35) (85,55) (88,35) 
\qbezier (35,35) (85,15) (88,35) 
\put(95,35){\makebox(0,0){$-$}}

\put(130,35){\vector(1,0){20}}                     
\put(130,35){\vector(-1,0){20}}
\put(130,45){\makebox(0,0){R7}}

\put(180,35){\circle*{4}}                         
\put(210,35){\circle*{4}}
\put(240,35){\circle*{4}}
\put(270,35){\circle*{4}}
\put(300,35){\circle*{4}}
\put(330,35){\circle*{4}}
\put(360,35){\circle*{4}}
\put(390,35){\circle*{4}}
\put(240,65){\circle*{4}}

\put(180,35){\line(1,0){30}}                      
\put(210,35){\line(1,0){30}}
\put(240,35){\line(1,0){30}}
\put(270,35){\line(1,0){30}}
\put(240,35){\line(0,1){30}}
\put(270,35){\line(1,0){30}}
\put(300,35){\line(1,0){30}}
\put(330,35){\line(1,0){30}}
\put(360,35){\line(1,0){30}}

\put(180,45){\makebox(0,0){$-2$}}                 
\put(210,45){\makebox(0,0){$-2$}}
\put(230,45){\makebox(0,0){$-2$}}
\put(270,45){\makebox(0,0){$-2$}}
\put(300,45){\makebox(0,0){$-2$}}
\put(230,65){\makebox(0,0){$-2$}}
\put(330,45){\makebox(0,0){$-2$}}	
\put(360,45){\makebox(0,0){$-2$}}	
\put(390,45){\makebox(0,0){$-2$}}

\end{picture}

\section{An orientation preserving diffeomorphisms of the boundaries of the
Dolbeault and Betti spaces via plumbing calculus}\label{sec:diffeo}

First we concentrate on the first line of the diagram of Theorem \ref{thm:Simpson}.

Note that $ \Mod_{\Dol}^X \setminus h^{-1} (B_R(0)) $ is diffeomorphic (by an orientation preserving diffeomorphism) to $(0,1]\times h^{-1}(S^1_R)$, where $S^1_R=\partial \overline{B_R(0)}$.
Hence, by the discussion from subsection \ref{subsec:Dolbeault},
$  \Mod_{\Dol}^X \setminus h^{-1} (B_R(0)) \simeq (0,1]\times \tilde{h}^{-1}(S^1_R)$.

On the other hand, $\Mod_{\Betti}^X\cap T^X\simeq (0,1]\times (\Mod_{\Betti}^X\cap \partial T^X)$.
Then, we need to prove that the corresponding boundaries $Y_{\Dol}^X:=\tilde{h}^{-1}(S^1_R)$ and
$Y_{\Betti}^X:=\Mod_{\Betti}^X\cap \partial T^X$ are diffeomorphic by an orientation preserving diffeomorphism.
Since both sides have their (oriented) plumbing graphs, the needed diffeomorphism
 will  be verified by plumbing calculus.

Next, we present  the corresponding sequences of moves of the  plumbing calculus,
for each  Painlev\'e cases. We will content ourselves
with describing the cases $X = I$ and $X=III(D6)$, because the remaining cases arise from one of these two cases by relatively straightforward modifications.

We start by the case $X = I$: according to Table~\ref{table}, we need to start at the graph corresponding to the affine root system $E_8^{(1)}$, and
get to the corresponding decorated graph pictured in Section~\ref{subsec:Betti}, using the moves given in Section~\ref{sec:plumbing}.
Here is a sequence of moves connecting these plumbing diagrams:

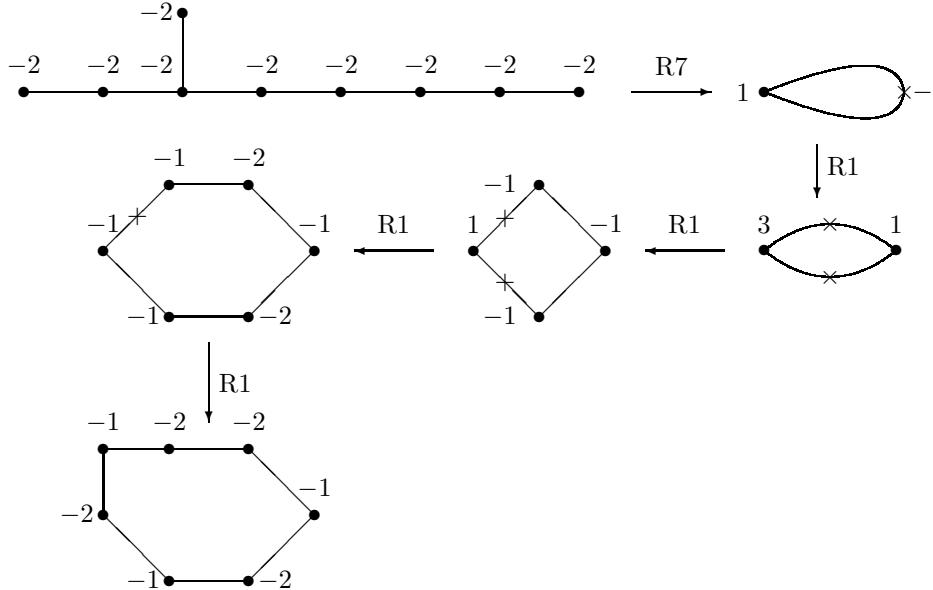
\begin{picture}(300,250)(0,-50)


\put(0,150){\circle*{4}}                         
\put(30,150){\circle*{4}}
\put(60,150){\circle*{4}}
\put(90,150){\circle*{4}}
\put(120,150){\circle*{4}}
\put(150,150){\circle*{4}}
\put(180,150){\circle*{4}}
\put(210,150){\circle*{4}}
\put(60,180){\circle*{4}}

\put(0,150){\line(1,0){30}}                      
\put(30,150){\line(1,0){30}}
\put(60,150){\line(1,0){30}}
\put(90,150){\line(1,0){30}}
\put(60,150){\line(0,1){30}}
\put(120,150){\line(1,0){30}}
\put(150,150){\line(1,0){30}}
\put(180,150){\line(1,0){30}}

\put(0,160){\makebox(0,0){$-2$}}                 
\put(30,160){\makebox(0,0){$-2$}}
\put(50,160){\makebox(0,0){$-2$}}
\put(90,160){\makebox(0,0){$-2$}}
\put(120,160){\makebox(0,0){$-2$}}
\put(50,180){\makebox(0,0){$-2$}}
\put(150,160){\makebox(0,0){$-2$}}	
\put(180,160){\makebox(0,0){$-2$}}	
\put(210,160){\makebox(0,0){$-2$}}


\put(230,150){\vector(1,0){30}}
\put(245,160){\makebox(0,0){R7}}


\put(280,150){\circle*{4}} 
\put(272,150){\makebox(0,0){$1$}}
\qbezier (280,150) (330,170) (333,150) 
\qbezier (280,150) (330,130) (333,150) 
\put(340,150){\makebox(0,0){$-$}}
\put(333,150){\makebox(0,0){$\times$}}


\put(300,130){\vector(0,-1){20}}
\put(310,122){\makebox(0,0){R1}}


\put(280,90){\circle*{4}} 
\put(330,90){\circle*{4}} 
\qbezier(280,90)(305,110)(330,90) 
\qbezier(280,90)(305,70)(330,90) 
\put(280,100){\makebox(0,0){$3$}}
\put(330,100){\makebox(0,0){$1$}}
\put(305,100){\makebox(0,0){$\times$}}
\put(305,80){\makebox(0,0){$\times$}}


\put(265,90){\vector(-1,0){30}}
\put(250,100){\makebox(0,0){R1}}


\put(220,90){\circle*{4}} 
\put(170,90){\circle*{4}}
\put(195,65){\circle*{4}}
\put(195,115){\circle*{4}}

\put(170,90){\line(1,1){25}} 
\put(220,90){\line(-1,1){25}}
\put(170,90){\line(1,-1){25}}
\put(220,90){\line(-1,-1){25}}

\put(220,100){\makebox(0,0){$-1$}}
\put(170,100){\makebox(0,0){$1$}}
\put(180,65){\makebox(0,0){$-1$}}
\put(180,115){\makebox(0,0){$-1$}}

\put(182,102){\makebox(0,0){$+$}}
\put(182,78){\makebox(0,0){$+$}}


\put(155,90){\vector(-1,0){30}}
\put(140,100){\makebox(0,0){R1}}


\put(110,90){\circle*{4}} 
\put(85,115){\circle*{4}}
\put(85,65){\circle*{4}}
\put(55,115){\circle*{4}}
\put(55,65){\circle*{4}}
\put(30,90){\circle*{4}}

\put(110,90){\line(-1,1){25}}
\put(110,90){\line(-1,-1){25}}
\put(85,115){\line(-1,0){30}}
\put(85,65){\line(-1,0){30}}
\put(55,115){\line(-1,-1){25}}
\put(55,65){\line(-1,1){25}}

\put(110,100){\makebox(0,0){$-1$}}
\put(85,125){\makebox(0,0){$-2$}}
\put(95,65){\makebox(0,0){$-2$}}
\put(55,125){\makebox(0,0){$-1$}}
\put(45,65){\makebox(0,0){$-1$}}
\put(30,100){\makebox(0,0){$-1$}}

\put(43,103){\makebox(0,0){$+$}}


\put(70,55){\vector(0,-1){30}}
\put(80,40){\makebox(0,0){R1}}


\put(110,-10){\circle*{4}} 
\put(85,15){\circle*{4}}
\put(85,-35){\circle*{4}}
\put(55,15){\circle*{4}}
\put(55,-35){\circle*{4}}
\put(30,-10){\circle*{4}}
\put(30,15){\circle*{4}}

\put(110,-10){\line(-1,1){25}}
\put(110,-10){\line(-1,-1){25}}
\put(85,15){\line(-1,0){30}}
\put(85,-35){\line(-1,0){30}}
\put(55,-35){\line(-1,1){25}}
\put(30,15){\line(1,0){25}}
\put(30,15){\line(0,-1){25}}

\put(110,0){\makebox(0,0){$-1$}}
\put(85,25){\makebox(0,0){$-2$}}
\put(95,-35){\makebox(0,0){$-2$}}
\put(55,25){\makebox(0,0){$-2$}}
\put(45,-35){\makebox(0,0){$-1$}}
\put(20,-10){\makebox(0,0){$-2$}}
\put(30,25){\makebox(0,0){$-1$}}

\end{picture}

We now come to the case $X=III(D6)$: we provide a sequence of plumbing calculus moves that allows one to pass from the graph corresponding to the affine root system $D_6^{(1)}$ into the
decorated graph given in Section~\ref{subsec:Betti}.

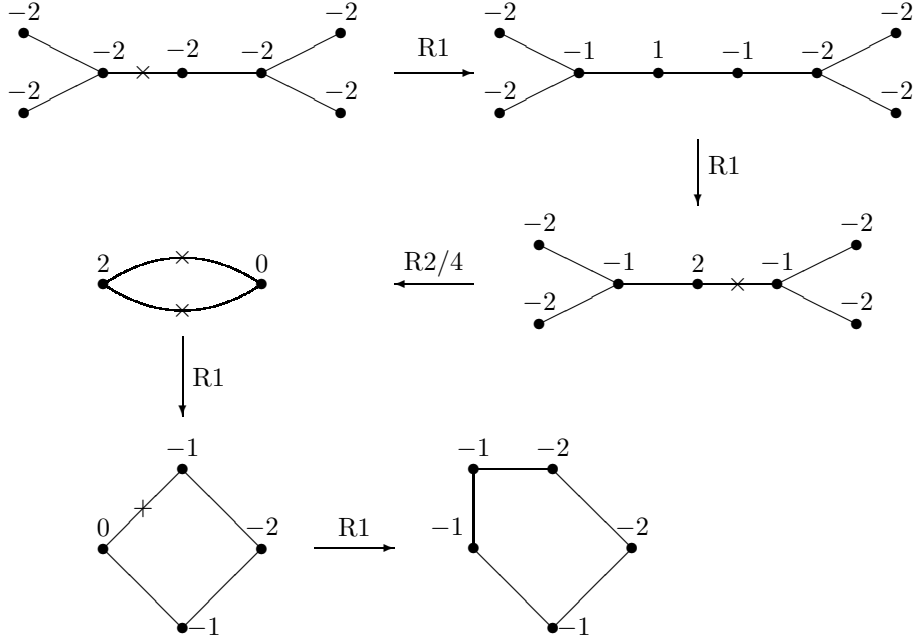
\begin{picture}(300,270)(10,-20)

\put(40,200){\circle*{4}}
\put(70,200){\circle*{4}}
\put(100,200){\circle*{4}}                           
\put(10,185){\circle*{4}}
\put(10,215){\circle*{4}}
\put(130,185){\circle*{4}}
\put(130,215){\circle*{4}}

\put(40,200){\line(1,0){60}}                        
\put(40,200){\line(-2,-1){30}}
\put(40,200){\line(-2, 1){30}}
\put(100,200){\line(2,-1){30}}
\put(100,200){\line(2, 1){30}}

\put(42,208){\makebox(0,0){$-2$}}                  
\put(70,209){\makebox(0,0){$-2$}}
\put(98,208){\makebox(0,0){$-2$}}
\put(10,193){\makebox(0,0){$-2$}}
\put(10,223){\makebox(0,0){$-2$}}
\put(130,193){\makebox(0,0){$-2$}}
\put(130,223){\makebox(0,0){$-2$}}

\put(55,200){\makebox(0,0){$\times$}}


\put(150,200){\vector(1,0){30}}
\put(165,210){\makebox(0,0){R1}}

\put(220,200){\circle*{4}}
\put(250,200){\circle*{4}}
\put(280,200){\circle*{4}}                            
\put(310,200){\circle*{4}}
\put(190,185){\circle*{4}}
\put(190,215){\circle*{4}}
\put(340,185){\circle*{4}}
\put(340,215){\circle*{4}}

\put(220,200){\line(1,0){90}}                        
\put(220,200){\line(-2,-1){30}}
\put(220,200){\line(-2, 1){30}}
\put(310,200){\line(2,-1){30}}
\put(310,200){\line(2, 1){30}}

\put(220,208){\makebox(0,0){$-1$}}                  
\put(250,209){\makebox(0,0){$1$}}
\put(280,208){\makebox(0,0){$-1$}}
\put(310,208){\makebox(0,0){$-2$}}
\put(190,193){\makebox(0,0){$-2$}}
\put(190,223){\makebox(0,0){$-2$}}
\put(340,193){\makebox(0,0){$-2$}}
\put(340,223){\makebox(0,0){$-2$}}


\put(265,175){\vector(0,-1){25}}
\put(275,165){\makebox(0,0){R1}}


\put(235,120){\circle*{4}}
\put(265,120){\circle*{4}}
\put(295,120){\circle*{4}}                           
\put(205,105){\circle*{4}}
\put(205,135){\circle*{4}}
\put(325,105){\circle*{4}}
\put(325,135){\circle*{4}}

\put(235,120){\line(1,0){60}}                        
\put(235,120){\line(-2,-1){30}}
\put(235,120){\line(-2, 1){30}}
\put(295,120){\line(2,-1){30}}
\put(295,120){\line(2, 1){30}}

\put(235,128){\makebox(0,0){$-1$}}                  
\put(265,128){\makebox(0,0){$2$}}
\put(295,128){\makebox(0,0){$-1$}}
\put(205,113){\makebox(0,0){$-2$}}
\put(205,143){\makebox(0,0){$-2$}}
\put(325,113){\makebox(0,0){$-2$}}
\put(325,143){\makebox(0,0){$-2$}}
\put(280,120){\makebox(0,0){$\times$}}


\put(180,120){\vector(-1,0){30}}
\put(165,128){\makebox(0,0){R2/4}}


\put(40,120){\circle*{4}}
\put(100,120){\circle*{4}}

\qbezier(40,120)(70,140)(100,120)
\qbezier(40,120)(70,100)(100,120)

\put(40,128){\makebox(0,0){$2$}}
\put(100,128){\makebox(0,0){$0$}}
\put(70,130){\makebox(0,0){$\times$}}
\put(70,110){\makebox(0,0){$\times$}}


\put(70,100){\vector(0,-1){30}}
\put(80,85){\makebox(0,0){R1}}


\put(40,20){\circle*{4}}       
\put(100,20){\circle*{4}}
\put(70,50){\circle*{4}}
\put(70,-10){\circle*{4}}

\put(40,20){\line(1,1){30}}        
\put(40,20){\line(1,-1){30}}
\put(100,20){\line(-1,1){30}}
\put(100,20){\line(-1,-1){30}}

\put(40,28){\makebox(0,0){$0$}}
\put(100,28){\makebox(0,0){$-2$}}
\put(70,58){\makebox(0,0){$-1$}}
\put(78,-10){\makebox(0,0){$-1$}}
\put(55,35){\makebox(0,0){$+$}}


\put(120,20){\vector(1,0){30}}
\put(135,28){\makebox(0,0){R1}}


\put(180,20){\circle*{4}}       
\put(240,20){\circle*{4}}
\put(210,50){\circle*{4}}
\put(210,-10){\circle*{4}}
\put(180,50){\circle*{4}}

\put(180,50){\line(1,0){30}}
\put(180,50){\line(0,-1){30}}
\put(210,50){\line(1,-1){30}}
\put(210,-10){\line(1,1){30}}
\put(210,-10){\line(-1,1){30}}

\put(170,28){\makebox(0,0){$-1$}}
\put(180,58){\makebox(0,0){$-1$}}
\put(210,58){\makebox(0,0){$-2$}}
\put(240,28){\makebox(0,0){$-2$}}
\put(218,-10){\makebox(0,0){$-1$}}

\end{picture}

\section{Proof of Theorem~\ref{thm:Simpson}}\label{sec:proof}
Since $|{\mathcal N}^X|=S^1$, we wish to prove that there exists a homotopy commutative diagram
$$
  \xymatrix{Y^X_{\Dol} \ar[d]_{h} \ar[r]^{\psi} & Y_{\Betti}^X \ar[d]^{\phi} \\
   S^1_R \ar[r] & S^1 }
$$
As the diffeomorphism $\psi:Y^X_{\Dol}\to Y^X_{\Betti}$ was already established, we focus on the corresponding
maps to $S^1$. If $Y$ is a connected oriented 3-manifold, then the homotopy classes of maps
$Y\to S^1$ are classified by 
$$
  [Y, S^1]=[Y, K(\Z,1)]=H^1(Y,\Z)={\rm Hom}(H_1(Y,\Z),\Z).
$$
 Hence we need to determine $H_1(Y, \Z)$ in both cases (though they are isomorphic).

 In general, if $Y$ is a plumbed 3-manifold with plumbing graph $\Gamma$ (and vertices ${\mathcal V}$),
 let $\{I_{v,u}\}_{v,u\in{\mathcal V}}$ be its  intersection matrix, $g=\sum_{v\in{\mathcal V}}g_v$
 be the sum of genera of the surfaces used in the plumbing
 (in all our cases this is zero), and $b_1(|{\mathcal N}^\Gamma|)$ the number of independent 1-cycles
 in the topological realization of $\Gamma$.
 (Recall that when $\Gamma$ has  no loops and all edge decorations are positive, then
 $I_{v,v}$ is the Euler number of $v$, otherwise $I_{v,u}$ ($u\not=v$)
 is the number of edges connecting $v$ and $u$.)
 Then, it is well-known (as follows from a Mayer-Vietoris argument) that
 $$H_1(Y,\Z)={\rm coker}(I)\oplus \Z^{2g+ b_1(|{\mathcal N}^\Gamma|)}, $$
 hence the first Betti number of $Y$ is
 $$b_1(Y)={\rm rank}\,{\ker}(I)+2g+ b_1(|{\mathcal N}^\Gamma|).$$
 Hence in our cases we get that $b_1(Y^X_{\Dol})=b_1(Y^X_{\Betti})=1$. Hence in both cases
 $[Y,S^1]=H^1(Y,\Z)=\Z$. Since the orientation of $|{\mathcal N}^X|$ (of the target of $\phi$) is not well-defined
 (depends on a choice of cyclic ordering of vertices in the dual graph) the homotopy diagram can be realized  only up to a sign of the isomorphism 
 \begin{align*}
  H^1(S^1_R,\Z)=\Z & \to H^1(|{\mathcal N}^X|,\Z)=\Z \\
  1 & \mapsto \pm 1 .
 \end{align*}
  Hence, it is enough to prove that at first-homology level both maps $h$ and $\phi$ induce epimorphisms
  $H^1(Y,\Z)\twoheadrightarrow \Z$. But this property is automatically satisfied in both cases,
  since both maps $h$ and $\phi$ are torus bundles over $S^1$:
  in the  first case the bundle structure is induced by the elliptic fibration, while in the second case
  by the specific plumbing construction of the cyclic  graphs, see e.g.
  \cite{nakamura,Neu,Nemethi_konyv} (where even the monodromy action is computed in terms of the Euler numbers).
  Since the fiber (the torus) is connected, surjectivity follows in both cases (say, from the long homotopy exact sequence).

  We emphasize, that in this case we have even a genuine (not only homotopy) identification of the
  torus bundles of $h$ and $\phi$ (up to an orientation ambiguity of the base spaces $S^1$).

\section{Discussion of Mixed Hodge Structures on the homology groups}

Assume that $S$ is a smooth compact complex projective surface  and $C\subset S$ is a normal crossing curve
with smooth irreducible components $\{C_v\}_{v\in {\mathcal V}}$. Then the boundary $Y$ of a small
tubular neighborhood of $C$ in $S$ admits a plumbing representation, where
the $C$-components are identified with the vertices and the intersection matrix of the components with
the intersection matrix associated with the (dual) graph. Furthermore, the complex (co)homology of $Y$
supports a Mixed Hodge structure. Following~\cite[Section~6.9]{EN} we have the following facts
(where $H_i$ denotes $H_i(Y,\C)$)
$$\dim\,{\rm Gr}^W_{-2}H_1={\rm rank}\, \ker(I), \ \ \dim\,{\rm Gr}^W_{-1}H_1=2g, \ \
\dim\,{\rm Gr}^W_{0}H_1=b_1(|{\mathcal N}^\Gamma|);$$
$$\dim\,{\rm Gr}^W_{-2}H_2={\rm rank}\, \ker(I), \ \ \dim\,{\rm Gr}^W_{-3}H_2=2g, \ \
\dim\,{\rm Gr}^W_{-4}H_2=b_1(|{\mathcal N}^\Gamma|).$$
In particular,
$$
    H_1 ( Y^X_{\Dol}, \C )  = \mbox{Gr}_{-2}^W H_1 ( Y^X_{\Dol}, \C ) \ \ \mbox{and} \ \
    H_1 ( Y^X_{\Betti}, \C )  =  \mbox{Gr}_0^W H_1 ( Y^X_{\Betti}, \C );
$$
$$
    H_2 ( Y^X_{\Dol}, \C )  = \mbox{Gr}_{-2}^W H_2 ( Y^X_{\Dol}, \C ) \ \ \mbox{and} \ \
    H_2 ( Y^X_{\Betti}, \C )  =  \mbox{Gr}_{-4}^W H_2 ( Y^X_{\Betti}, \C ).
$$
Note that from the topology (the oriented diffeomorphism type) of the corresponding plumbed 3-manifolds one cannot read the Mixed Hodge structure of the boundary of a tubular neighborhood of a  projective plane, for this one needs
data from the algebraic realization. Indeed, the above isomorphisms illustrate the fact that the plumbing calculus might mix up the three
types of contributions, which for the MHS are essentially distinct. (For similar examples see e.g.~\cite{SS}.)

Notice that $\Mod^X_{\Dol}$ has $h^{-1} (B_R(0))$ as deformation retract, so for topological purposes we may identify these spaces. 
Consider next the pairs $Y^X_{\Dol}=\partial \Mod^X_{\Dol}\subset \Mod^X_{\Dol}$ and  
$Y^X_{\Betti}=\partial \Mod^X_{\Betti}\subset \Mod^X_{\Betti}$, abbreviated uniformly as 
$Y=\partial \Mod\subset \Mod$. Then one has (part of) the homology exact sequence (with complex coefficients)
$$\cdots \to H_3(\Mod, Y)\to H_2(Y)\to H_2(\Mod)\to \cdots$$
Here $\Mod$ is an oriented smooth 4-manifold with boundary $Y$, hence by Lefschetz duality 
$  H_3(\Mod, Y)=H_1(\Mod)^*$. Note that 
$H_1(\Mod^X_{\Betti})=H_1(\Mod^X_{\Dol})=0$ (see~\cite[Lemma~3]{Sz_PW}), 
hence in both cases, $\C=H_2(Y)$ embeds into $H_2(\Mod)$. Since $\C=H_2(Y)$
in the two cases supports different MHS weight (and the morphism preserves the weights),
the MHS in $H_2$ of the moduli spaces $\Mod$ will also be different. 
The responsible objects for this MHS difference are the different plumbing graphs (which can 
be identified by plumbing calculus, but their data needed for Hodge considerations are not the same). 

This provides a direct geometric explanation in our cases of the well-known change of weights between the Dolbeault and the Betti spaces. 
Namely, the weight of the generator of $H_2 ( Y^X_{\Betti}, \C )$ shows that $H_2 (\Mod^X_{\Betti}, \C )$ is not pure, while $H_2 (\Mod^X_{\Dol}, \C )$ is pure. 
On the other hand, as it was already pointed out in~\cite{Sz_PW}, it follows from the geometric characterization of the perverse filtration provided in 
\cite[Theorem~4.1.1]{dCM} in terms of the flag filtration that the weight of the generator of $H_2 ( Y^X_{\Dol}, \C )$ (the generic Hitchin fiber) in 
$H_2(\Mod^X_{\Dol})$ has perverse Leray degree different from the classes in $H_2(\Mod^X_{\Dol}, Y^X_{\Dol})$.

\end{document}